\documentclass[11pt,reqno]{amsart}
\usepackage[T1]{fontenc}
\usepackage{lmodern}
\usepackage[margin=1in]{geometry}
\usepackage{mathtools,amssymb,amsthm}
\usepackage{enumitem,needspace}
\usepackage[expansion=false]{microtype}
\usepackage{cite}
\usepackage[hidelinks]{hyperref}
\hypersetup{pdftitle={Convex Bernstein limits and sharp recovery on expanding intervals},
 pdfauthor={Matthew J. Colbrook and Catherine Drysdale}}
\numberwithin{equation}{section}
\newtheorem{theorem}{Theorem}[section]
\newtheorem{lemma}[theorem]{Lemma}
\newtheorem{proposition}[theorem]{Proposition}
\newtheorem{corollary}[theorem]{Corollary}
\theoremstyle{definition}
\newtheorem{example}[theorem]{Example}
\theoremstyle{remark}
\newtheorem{remark}[theorem]{Remark}
\newcommand{\R}{\mathbb R}
\newcommand{\C}{\mathbb C}

\newcommand{\Pn}{\mathcal P_n}
\newcommand{\B}{\mathcal B}
\newcommand{\Bc}{\mathcal B^{(+2)}}
\newcommand{\Ec}{E^{(+2)}}
\newcommand{\Ac}{A^{(+2)}}
\newcommand{\betac}{\beta^{(+2)}}
\newcommand{\norm}[2]{\left\|#1\right\|_{#2}}

\newcommand{\arcosh}{\operatorname{arcosh}}
\newcommand{\arsinh}{\operatorname{arsinh}}
\newcommand{\Real}{\operatorname{Re}}
\newcommand{\Imag}{\operatorname{Im}}

\newcommand{\dd}{\,\mathrm d}
\setlist[enumerate]{leftmargin=2.4em,itemsep=3pt,topsep=4pt}

\title[Convex Bernstein limits and sharp recovery]{Convex Bernstein limits and\\sharp recovery on expanding intervals}
\author[Matthew J. Colbrook]{Matthew J. Colbrook}
\address[Matthew J. Colbrook]{Department of Applied Mathematics and
Theoretical Physics, University of Cambridge, United Kingdom}
\email{m.colbrook@damtp.cam.ac.uk}
\author[Catherine Drysdale]{Catherine Drysdale}
\address[Catherine Drysdale]{School of Mathematical Sciences, Lancaster University,
United Kingdom}
\email{c.drysdale@lancaster.ac.uk}
\thanks{Corresponding author: Matthew J. Colbrook.}
\date{}

\begin{document}
\begin{abstract}
We determine when approximation by convex polynomials on expanding intervals converges to best approximation on the real line by convex entire functions of exponential type at most $\sigma>0$. For intervals of radius $(n-\gamma_n)/\sigma$, where $0\le\gamma_n<n$ and $\gamma_n=o(n)$, one margin sequence gives this limit for every
convex function of polynomial growth if and only if $\gamma_n/[n^{1/3}(\log n)^{2/3}]\to\infty$. When this condition fails, a fixed uniformly strictly convex entire function has approximation errors with unbounded limsup. We apply the limit theorem to prove existence of the convex Bernstein constant for $|x|^\lambda$ for every $\lambda\ge1$, and identify it with an attained entire-function minimum. We also give a sufficient margin
for growth bounded by $\exp(a|x|^\alpha)$, $0<\alpha<1$. For arbitrary subexponential growth on the real line, each target with finite entire-function error admits a suitable margin, whereas no single margin is sufficient for the whole class. Further results concern attainment, right continuity with respect to the exponential type,
weighted growth conditions, and the endpoints of fixed-fraction approximation intervals.
\end{abstract}
\maketitle

\section{Introduction and main results}\label{sec:intro}

Bernstein's limit theorem relates polynomial approximation of power
functions on a fixed interval to approximation on $\R$ by entire
functions of exponential type
\cite{Bernstein1914,Bernstein1938,Lubinsky2007}. We study the
corresponding limit problem under the convexity constraint. The main
question is how closely the approximation interval can approach the
critical radius $n/\sigma$ while preserving this relation for a whole
class of target functions.

For each integer $n\ge0$, let $\Pn$ denote the real algebraic
polynomials of degree at most $n$.
For a compact interval $I$ and a real continuous function $f$ on $I$, define
\begin{align}
 E_n(f;I)&=\inf_{p\in\Pn}\norm{f-p}{I},\label{eq:En}\\
 \Ec_n(f;I)&=\inf_{\substack{p\in\Pn\\p''\ge0\text{ on }I}}
                         \norm{f-p}{I},\label{eq:Econv}
\end{align}
where $\norm{\cdot}{I}$ is the uniform norm. The superscript $(+2)$ denotes
convexity. For $\sigma\ge0$, let $\B_\sigma$ consist of the entire functions
that are real on $\R$ and have exponential type at most $\sigma$: for every
$\varepsilon>0$,
\[
 |F(z)|\le C_\varepsilon e^{(\sigma+\varepsilon)|z|},\qquad z\in\C.
\]
The convex subclass is
$\Bc_\sigma=\{F\in\B_\sigma:F''\ge0\text{ on }\R\}$, and its best
whole-line approximation error is
\begin{equation}\label{eq:Ac}
 \Ac_\sigma(f)=\inf_{F\in\Bc_\sigma}\norm{f-F}{\R}\in[0,\infty].
\end{equation}
The unconstrained error $A_\sigma(f)$ is defined by replacing
$\Bc_\sigma$ with $\B_\sigma$. A finite value in \eqref{eq:Ac} does not
require $f$ itself to be bounded. In a polynomial-growth bound we always
take $C>0$ and $M\ge0$; stretched-exponential envelopes have $C,a>0$.

Uniform approximation of a convex function does not control the sign
of the approximating polynomial's second derivative. Nevertheless,
convex approximation preserves every positive algebraic rate available
without the constraint \cite{KLS2005,KLS2006,KLPS2011}. This comparison
bounds the normalized errors for power functions, but the existence of
their limit requires a further argument.

Gal asked whether $n^\lambda\Ec_n(|x|^\lambda;[-1,1])$ has a finite limit
for every $\lambda\ge1$, and raised a more general question about convex
approximation on intervals approaching $[-n/\sigma,n/\sigma]$
\cite{Gal2016}.
His note traces the first question to \cite[p.~325]{Gal2009}. The second
uses intervals $[-n(1-\lambda_n)/\sigma,n(1-\lambda_n)/\sigma]$ and refers
to \cite[p.~3]{Ganzburg2008} for conditions on $(\lambda_n)$.
For $\lambda=1$, Foucart and Powers
\cite[Sect.~4]{FoucartPowers2017} state the existence of the limit and
study it numerically. Theorem~\ref{thm:bernstein} proves existence for
every $\lambda\ge1$ and identifies the limit with an attained
entire-function minimum.

The central question of this paper is the recovery identity
\begin{equation}\label{eq:recovery}
 \lim_{n\to\infty}\Ec_n(f;[-R_n,R_n])=\Ac_\sigma(f),
 \qquad R_n=\frac{n-\gamma_n}{\sigma}.
\end{equation}
Unless stated otherwise, $\sigma>0$ is fixed and
\begin{equation}\label{eq:margin}
 0\le\gamma_n<n,\qquad \gamma_n=o(n).
\end{equation}
All logarithms are natural. Finitely many initial indices have no bearing
on the assertions below. Set
\begin{equation}\label{eq:Gn}
 G_n=n^{1/3}(\log n)^{2/3},\qquad c_0=\frac{2\sqrt2}{3}.
\end{equation}

\begin{theorem}[Sharp universal polynomial-growth recovery]\label{thm:main}
Under \eqref{eq:margin}, the following conditions are equivalent.
\begin{enumerate}[label=\textup{(\roman*)}]
\item $\gamma_n/G_n\to\infty$.
\item For every $F\in\Bc_\sigma$ satisfying
$|F(x)|\le C(1+|x|)^M$ on $\R$ for some $C>0$ and $M\ge0$,
$\Ec_n(F;[-R_n,R_n])\to0$.
\item Identity \eqref{eq:recovery} holds, in the extended sense, for every
continuous convex $f$ satisfying $|f(x)|\le C(1+|x|)^M$ on $\R$ for some
finite $C,M$.
\end{enumerate}
If \textup{(i)} fails, there is a single even, uniformly strictly convex
$F\in\Bc_\sigma$ of polynomial real-axis growth such that
$\Ac_\sigma(F)=0$ and $\limsup_n\Ec_n(F;[-R_n,R_n])=\infty$.
\end{theorem}

The same margin sequence must work for \emph{all finite growth orders}:
the exponent $M$ is fixed for each target and may vary across the class.
This quantifier determines the scale in Theorem~\ref{thm:main}.
In terms of the relative margin $\delta_n=\gamma_n/n$, the condition is
\[
 \delta_n\to0,\qquad
 \frac{n\delta_n^{3/2}}{\log n}\to\infty.
\]

Example~\ref{ex:quartic} gives a counterexample of quartic growth when
$\gamma_n=G_n$.

For power functions, homogeneity converts recovery on expanding
intervals into a limit on $[-1,1]$. Together with the algebraic-rate
bound for convex approximation, this gives the following result.

\begin{theorem}[Convex Bernstein constants]\label{thm:bernstein}
For every $\lambda\ge1$,
\begin{equation}\label{eq:bernstein}
 \betac_\lambda:=\lim_{n\to\infty}
 n^\lambda\Ec_n(|x|^\lambda;[-1,1])
 =\min_{F\in\Bc_1}\norm{|x|^\lambda-F}{\R}<\infty.
\end{equation}
An extremizer may be chosen even. The constant is zero precisely when
$\lambda$ is a positive even integer, and is positive otherwise.
\end{theorem}

The size of the required margin also depends on the growth allowed
along the real axis. We next consider a fixed stretched-exponential
envelope and then the full class of subexponential growth:
\begin{align}
 |f(x)|&\le C e^{a|x|^\alpha},\qquad 0<\alpha<1,
                         \label{eq:stretch-growth}\\
 \text{for every }\varepsilon>0,
 \quad |f(x)|&\le C_\varepsilon e^{\varepsilon|x|}.
                         \label{eq:subexp-growth}
\end{align}
Condition \eqref{eq:subexp-growth} is equivalent to
$\log^+|f(x)|=o(|x|)$ as $|x|\to\infty$. Both conditions concern
the restriction of the target to $\R$; its entire approximants may have
positive exponential type.

\Needspace{8\baselineskip}
\begin{theorem}[Stretched-exponential recovery]\label{thm:stretch}
Let $0<\alpha<1$, and suppose that \eqref{eq:margin} holds and
\begin{equation}\label{eq:stretch-margin}
 \frac{\gamma_n}{n^{(1+\alpha)/(3-\alpha)}}\longrightarrow\infty.
\end{equation}
Then \eqref{eq:recovery} holds, in the extended sense, for every continuous
convex $f$ satisfying \eqref{eq:stretch-growth}.
\end{theorem}

\begin{theorem}[Subexponential dichotomy]\label{thm:dichotomy}
Fix $\sigma>0$.
\begin{enumerate}[label=\textup{(\roman*)}]
\item Every continuous convex $f$ satisfying \eqref{eq:subexp-growth} and
$\Ac_\sigma(f)<\infty$ admits a sequence \eqref{eq:margin} for which
\eqref{eq:recovery} holds. The sequence can be chosen with
$\gamma_n\to\infty$ and $\gamma_n/n$ nonincreasing to zero, but it may
depend on $f$ and $\sigma$.
\item For every prescribed sequence \eqref{eq:margin}, there is a single
even real entire function $F$ of type exactly $\sigma$, satisfying
\eqref{eq:subexp-growth} and $F''\ge1$ on $\R$, such that
\[
 \Ac_\sigma(F)=0,\qquad
 \limsup_{n\to\infty}\Ec_n(F;[-R_n,R_n])=\infty.
\]
\end{enumerate}
\end{theorem}

The two parts distinguish a margin chosen for a target from a margin
chosen for an entire growth class. The role of growth on the real line
is already visible for $F(x)=e^{\sigma x}$: its entire-function error is
zero, but its polynomial errors diverge for every sequence
\eqref{eq:margin}; see Proposition~\ref{prop:finite-difference}.
Theorem~\ref{thm:dichotomy}(ii) gives a counterexample even under
subexponential growth and uniform strict convexity.

\subsection{Relation to earlier work and organization}
Limit relations between polynomial and entire-function approximation
are developed in \cite{Ganzburg1991,Ganzburg2008,Ganzburg2023}.
Further Chebyshev-based estimates appear in
\cite[Sect.~3.1]{Ganzburg2025}. The classical construction for preserving
convexity approximates the second derivative, adds a nonnegative
constant, and integrates twice \cite{Shisha1965,KLPS2011,KLS2019}.
We combine this construction with directional growth estimates whose
exponential coefficient is $\sigma$. Optimizing the resulting bound
near the critical radius gives the sufficient margins. For polynomial
growth, polynomially weighted cosines give the sharp lower bound.
The weighted criterion in Appendix~\ref{app:weights} uses the growth
theorem of Berger--Schilling--Shargorodsky--Sharia \cite{BSSS2024}.

Section~\ref{sec:tools} establishes the growth estimates, convexity
correction, and compactness results used throughout the paper.
Section~\ref{sec:polynomial} proves Theorem~\ref{thm:main}, and
Section~\ref{sec:powers} applies it to power functions and their
extremizers. Sections~\ref{sec:stretch} and \ref{sec:subexp} treat the
two subexponential growth classes. Section~\ref{sec:type} studies the
dependence on exponential type. Appendix~\ref{app:weights} gives the
weighted criterion and a logarithmic growth envelope, while
Appendix~\ref{app:fractions} compares fixed-fraction scales and their
endpoints.

\section{Growth estimates and common approximation tools}\label{sec:tools}

The proofs use three ingredients: directional growth bounds without
an increase in exponential type, polynomial approximation followed by a convexity
correction, and compactness on expanding intervals. We establish these
ingredients in turn.

Constants may change from line to line and may depend on a fixed target,
growth parameters, and $\sigma$, but not on the polynomial degree. Local
uniform convergence always means convergence on every compact subset of
$\C$.

\subsection{Directional growth bounds}
The half-plane form of Bernstein's estimate converts a bounded
real-axis restriction into a directional growth bound with the same
exponential type \cite{Boas1954}.

\begin{lemma}[Half-plane estimate]\label{lem:halfplane}
Suppose $H$ is analytic in the upper half-plane, continuous on its
closure, and for every $\varepsilon>0$ satisfies
$|H(z)|\le C_\varepsilon e^{(\sigma+\varepsilon)|z|}$ there.
If $|H(x)|\le M$ for real $x$, then
\[
 |H(x+iy)|\le M e^{\sigma y},\qquad y\ge0.
\]
Consequently, if $H\in\B_\sigma$ is bounded on $\R$, then
$|H(x+iy)|\le\norm{H}{\R}e^{\sigma|y|}$.
\end{lemma}
\begin{proof}
Fix $\varepsilon>0$ and apply the Phragm\'en--Lindel\"of theorem in the
two upper quadrants to $H(z)e^{i(\sigma+\varepsilon)z}$. On the real
boundary it is bounded by $M$; on the positive imaginary axis it is
bounded, using the growth allowance $\varepsilon/2$. Its growth in either
quadrant is of order at most one, below the critical order two. It is
therefore bounded in the upper half-plane. The maximum principle for
bounded analytic functions in a half-plane, applied to its real boundary,
bounds its modulus by $M$. Let $\varepsilon\downarrow0$. Reflection gives
the assertion in the lower half-plane.
\end{proof}

\begin{lemma}[Polynomial and stretched envelopes]\label{lem:exactgrowth}
Let $F\in\B_\sigma$ and let $j$ be a fixed nonnegative integer.
\begin{enumerate}[label=\textup{(\roman*)}]
\item If $|F(x)|\le C(1+|x|)^M$ on $\R$, then
\begin{equation}\label{eq:polystrip}
 |F^{(j)}(z)|\le C_j(1+|z|)^M e^{\sigma|\Imag z|}.
\end{equation}
\item If $|F(x)|\le C e^{a|x|^\alpha}$, $0<\alpha<1$, then, with
$d=a/\cos(\pi\alpha/2)$,
\begin{equation}\label{eq:stretchstrip}
 |F^{(j)}(x+iy)|\le C_j
 \exp\bigl(\sigma|y|+d|x|^\alpha+d|y|^\alpha\bigr).
\end{equation}
\end{enumerate}
\end{lemma}
\begin{proof}
For (i), take an integer $k\ge\max\{1,M\}$ and $0<b\le1$. The function
\[
 H_b(z)=F(z)\left(\frac{\sin bz}{bz}\right)^k
\]
is bounded on $\R$ by $C b^{-M}$ and has type at most $\sigma+kb$.
Lemma~\ref{lem:halfplane} gives
$|H_b(z)|\le Cb^{-M}e^{(\sigma+kb)|\Imag z|}$.
At a specified $z$, set $b=[2(1+|z|)]^{-1}$. The sinc factor is bounded
away from zero for $|bz|\le1/2$, while $kb|\Imag z|\le k/2$.
Dividing gives \eqref{eq:polystrip} for $j=0$. Cauchy's formula on unit
circles gives every fixed derivative, with the same power $M$.

For (ii), in the upper half-plane use the principal branch to define
\[
 D_+(z)=\exp\bigl(d(1-iz)^\alpha\bigr).
\]
It is nonzero, $|D_+(z)|\ge1$ there, and
$|D_+(x)|\ge e^{a|x|^\alpha}$ on the boundary. Apply
Lemma~\ref{lem:halfplane} to $F/D_+$; division does not increase the
upper-half-plane type. Since
$|1+y-ix|^\alpha\le1+|x|^\alpha+|y|^\alpha$ for $y\ge0$, this gives
\eqref{eq:stretchstrip}, with an adjusted constant. The function
$D_-(z)=\exp(d(1+iz)^\alpha)$ gives the lower-half-plane bound.
Cauchy's formula and subadditivity of $t^\alpha$ give the derivative
estimates.
\end{proof}

Retaining the coefficient $\sigma$ of $|\Imag z|$ is essential at the
scales considered here. A fixed increase in type changes the critical
radius by a quantity of order $n$, while the available margin is $o(n)$.

\subsection{Chebyshev recovery and the convexity correction}
For $R,t>0$, denote by $\mathcal E(R,t)$ the Bernstein ellipse
\[
 z=R\cos(\theta-it)
   =R\cos\theta\cosh t+iR\sin\theta\sinh t,
 \qquad 0\le\theta\le2\pi.
\]
Let $H$ be entire and bounded by $M$ on $\mathcal E(R,t)$. Its
Chebyshev truncation $r_m$ of degree $m$ satisfies
\begin{equation}\label{eq:cheb}
 \norm{H-r_m}{[-R,R]}\le\frac{2M e^{-mt}}{e^t-1}.
\end{equation}
Indeed, the coefficients of order $k$ have modulus at most $2Me^{-kt}$,
and summing the tail gives \eqref{eq:cheb}; see also
\cite[Chapter~8]{Trefethen2019}. If $H$ is real on $\R$, then $r_m$ has real
coefficients.

To preserve convexity, we apply this estimate to the second derivative.
Adding its approximation error makes the approximating derivative
nonnegative, and two integrations produce a convex polynomial.

\begin{lemma}[Convex recovery]\label{lem:repair}
Let $n\ge2$, $R>0$, and $\delta\ge0$, and let $F\in C^2[-R,R]$ be convex.
If $r\in\mathcal P_{n-2}$ and
$\norm{F''-r}{[-R,R]}\le\delta$, the polynomial determined by
\begin{equation}\label{eq:repair}
 q''=r+\delta,\qquad q(0)=F(0),\qquad q'(0)=F'(0)
\end{equation}
has degree at most $n$, is convex on $[-R,R]$, and satisfies
\begin{equation}\label{eq:repairbound}
 0\le q(x)-F(x)\le\delta x^2\quad (|x|\le R).
\end{equation}
In particular, if $F$ is real entire and convex on $\R$, then for every
$t>0$,
\begin{equation}\label{eq:ellipse-recovery}
 \Ec_n(F;[-R,R])\le
 \frac{2R^2 e^{-(n-2)t}}{e^t-1}
 \max_{z\in\mathcal E(R,t)}|F''(z)|.
\end{equation}
\end{lemma}
\begin{proof}
We have $0\le q''-F''\le2\delta$, and hence $q''\ge0$.
The matched initial values give, for either sign of $x$,
\[
 q(x)-F(x)=x^2\int_0^1(1-s)
                   \bigl(q''(sx)-F''(sx)\bigr)\dd s.
\]
This proves \eqref{eq:repairbound}. Apply \eqref{eq:cheb} to
$F''$ with $m=n-2$ to obtain \eqref{eq:ellipse-recovery}.
\end{proof}

\subsection{Compactness on expanding intervals}
The lower bounds for the limit identities follow by passing from
polynomial approximants on expanding intervals to entire functions of
controlled exponential type.

Let
\[
 \Phi_R(z)=\frac zR+\sqrt{\frac{z^2}{R^2}-1}
\]
be the exterior conformal map of $\C\setminus[-R,R]$, with the branch
chosen so that $\Phi_R(z)\sim2z/R$ at infinity. On $[-R,R]$, the
notation $|\Phi_R|$ refers to its boundary value, which is one.
The Bernstein--Walsh
inequality is
\begin{equation}\label{eq:BW}
 |p(z)|\le\norm{p}{[-R,R]}|\Phi_R(z)|^{\deg p}.
\end{equation}
Uniformly for $z$ in a bounded set, with boundary modulus one on
$[-R,R]$,
\begin{equation}\label{eq:phi-limit}
 R\log|\Phi_R(z)|\longrightarrow|\Imag z|
 \quad\text{as }R\to\infty.
\end{equation}
Indeed, on a fixed bounded set and for sufficiently large $R$,
\[
 \log|\Phi_R(z)|=|\Imag\arccos(z/R)|
                 =\frac{|\Imag z|}{R}+O(R^{-3}),
\]
uniformly in $z$, by the Taylor expansion of $\arccos$ at zero.

\begin{lemma}[Weighted polynomial compactness]\label{lem:compactness}
Suppose $R_j\to\infty$, $n_j/R_j\to\sigma$, and
$p_j\in\mathcal P_{n_j}$. If, on $[-R_j,R_j]$, either
\begin{equation}\label{eq:compact-poly}
 |p_j(x)|\le C(1+|x|)^M
\end{equation}
with common finite $C,M$, or
\begin{equation}\label{eq:compact-stretch}
 |p_j(x)|\le C e^{a|x|^\alpha},\qquad 0<\alpha<1,
\end{equation}
with common $C,a,\alpha$, then $(p_j)$ is locally bounded in $\C$.
Every locally uniform subsequential limit belongs to $\B_\sigma$.
If each $p_j$ is convex on its interval, every such limit belongs to
$\Bc_\sigma$.
\end{lemma}
\begin{proof}
In the polynomial case choose an integer $s\ge\max\{1,M\}$. In the upper
half-plane apply the maximum principle to
\[
 \frac{p_j(z)}{(z+i)^s\Phi_{R_j}(z)^{n_j}}.
\]
On the portion of the boundary inside the interval its modulus is
bounded by a common constant. Outside the interval, \eqref{eq:BW} and
$\norm{p_j}{[-R_j,R_j]}\le C(1+R_j)^M$ give the same bound after division
by $|x+i|^s$. The quotient tends to zero at infinity; thus
\[
 |p_j(z)|\le C|z+i|^s|\Phi_{R_j}(z)|^{n_j}.
\]
Use $z-i$ in the lower half-plane. These bounds, including their
continuous boundary limits, give local boundedness across the real axis.
Equation~\eqref{eq:phi-limit} shows that a subsequential limit satisfies
$|F(z)|\le C(1+|z|)^s e^{\sigma|\Imag z|}$.

For \eqref{eq:compact-stretch}, use instead the quotient
$p_j/(D_+\Phi_{R_j}^{n_j})$, with $D_+$ as in the proof of
Lemma~\ref{lem:exactgrowth}. On the interval $D_+$ absorbs the given
weight. Outside it, \eqref{eq:BW} and
$|D_+(x)|\ge e^{a|x|^\alpha}\ge e^{aR_j^\alpha}$ give a common boundary
bound. Moreover $|D_+(z)|\to\infty$ uniformly as $z$ tends to infinity
in the closed upper half-plane, so the quotient vanishes there. The
corresponding lower-half-plane argument gives local boundedness and, for
a limit,
\[
 |F(z)|\le C\exp\bigl(\sigma|\Imag z|+d(1+|z|)^\alpha\bigr).
\]
Both bounds imply type at most $\sigma$. Montel's theorem gives
subsequences; real-valuedness passes to their limits. If the polynomials
are convex, locally uniform convergence also gives convergence of their
second derivatives, and every fixed real interval eventually lies in
their domains of convexity.
\end{proof}

\begin{corollary}[A weighted liminf inequality]\label{cor:weighted-liminf}
If a continuous convex $f$ satisfies a polynomial-growth bound or
\eqref{eq:stretch-growth}, then every sequence \eqref{eq:margin} satisfies
\begin{equation}\label{eq:weighted-liminf}
 \Ac_\sigma(f)\le\liminf_{n\to\infty}\Ec_n(f;[-R_n,R_n]).
\end{equation}
This includes the assertion that the liminf is infinite when
$\Ac_\sigma(f)=\infty$.
\end{corollary}
\begin{proof}
If the liminf is finite, choose convex polynomials along a subsequence
whose errors tend to it. Their real-axis bounds have common parameters
of the same growth class as $f$. Lemma~\ref{lem:compactness} yields a
convex entire limit of type at most $\sigma$. On each fixed real point
its error is bounded by the liminf. This proves the claim.
\end{proof}

\begin{proposition}[Attainment]\label{prop:attainment}
For any real-valued $f$ with $\Ac_\sigma(f)<\infty$, the infimum in
\eqref{eq:Ac} is attained. No growth assumption on $f$ is required.
\end{proposition}
\begin{proof}
Fix one finite-error competitor $F_0\in\Bc_\sigma$ and let $(F_j)$ be a
minimizing sequence. The functions $F_j-F_0$ have type at most $\sigma$
and a common bound on $\R$. Lemma~\ref{lem:halfplane} bounds them by
$K e^{\sigma|\Imag z|}$. A locally uniform subsequential limit, added to
$F_0$, belongs to $\Bc_\sigma$ and attains the infimum.
\end{proof}

\section{Sharp recovery for polynomial growth}\label{sec:polynomial}

\subsection{The upper estimate}
Optimizing the ellipse parameter leads to the function
\begin{equation}\label{eq:Psi}
 \Psi(r)=\arcosh(1/r)-\sqrt{1-r^2},\qquad 0<r<1.
\end{equation}
As $u\downarrow0$,
\begin{equation}\label{eq:Psi-expansion}
 \Psi(1-u)=c_0u^{3/2}(1+O(u)).
\end{equation}
To obtain \eqref{eq:Psi-expansion}, differentiate to find
$\Psi'(r)=-\sqrt{1-r^2}/r$ and integrate from $1-u$ to $1$.

\begin{proposition}[A quantitative estimate for polynomial growth]
\label{prop:poly-upper}
Let $F\in\Bc_\sigma$ satisfy $|F(x)|\le C(1+|x|)^M$. Assume
\eqref{eq:margin} and $\gamma_n\to\infty$. For all sufficiently large $n$,
define
\begin{equation}\label{eq:saddle-parameters}
 m=n-2,\quad b=n-\gamma_n,\quad
 t_n=\arcosh(m/b),\quad
 D_n=mt_n-\sqrt{m^2-b^2}.
\end{equation}
Then
\begin{align}
 \Ec_n(F;[-R_n,R_n])
 &\le C n^{M+2}(e^{t_n}-1)^{-1}e^{-D_n},\label{eq:poly-exact}\\
 &\le C n^{M+2}\sqrt{\frac n{\gamma_n}}
 \exp\left[-(c_0+o(1))\frac{\gamma_n^{3/2}}{\sqrt n}\right].
 \label{eq:poly-upper}
\end{align}
If $\gamma_n/G_n\to\infty$, this error is $o(n^{-K})$ for every fixed
$K>0$.
\end{proposition}
\begin{proof}
On $\mathcal E(R_n,t_n)$, we have
\[
 |z|\le R_n\cosh t_n=\frac m\sigma,\qquad
 \sigma|\Imag z|\le b\sinh t_n=\sqrt{m^2-b^2}.
\]
Apply \eqref{eq:polystrip} to $F''$, then
Lemma~\ref{lem:repair}, to obtain \eqref{eq:poly-exact}. Furthermore,
\[
 t_n\sim\sqrt{\frac{2\gamma_n}{n}},\qquad
 D_n=m\Psi(b/m)
     =(c_0+o(1))\frac{\gamma_n^{3/2}}{\sqrt n}.
\]
This proves \eqref{eq:poly-upper}. If $\gamma_n/G_n\to\infty$, the
quantity $\gamma_n^{3/2}/\sqrt n$ dominates $\log n$ by an unbounded
factor. It therefore absorbs every fixed polynomial prefactor.
\end{proof}

\begin{proof}[Proof of Theorem~\ref{thm:main}: sufficiency]
Proposition~\ref{prop:poly-upper} proves (i)$\Rightarrow$(ii). Suppose (ii)
and first consider a target in (iii) with finite whole-line error. Every
competitor in $\Bc_\sigma$ with finite uniform error inherits a
polynomial growth bound on $\R$.
The triangle inequality and (ii) therefore give
\[
 \limsup_n\Ec_n(f;[-R_n,R_n])\le\Ac_\sigma(f).
\]
Corollary~\ref{cor:weighted-liminf} gives the reverse inequality. If the
whole-line error is infinite, the same corollary gives the asserted
extended limit. Thus (ii)$\Rightarrow$(iii). Taking $f=F\in\Bc_\sigma$
shows (iii)$\Rightarrow$(ii). The necessity of (i) is proved below.
\end{proof}

\subsection{An oscillatory lower bound}
For the lower bound, we estimate a Chebyshev coefficient of
$x^{2k}\cos x$. This coefficient is expressed through a derivative of even order
of the Bessel integral. For an integer $m\ge0$, let
\begin{equation}\label{eq:J-integral}
 J_m(R)=\frac1{2\pi}\int_{-\pi}^{\pi}
                  e^{i(m\theta-R\sin\theta)}\dd\theta.
\end{equation}

\begin{lemma}[Differentiated saddle estimate]\label{lem:J-saddle}
Let $m\to\infty$ through positive integers, let $0<R<m$, and assume
\[
 \frac Rm\to1,\qquad \frac{m-R}{m^{1/3}}\to\infty.
\]
For every fixed nonnegative integer $q$, set
$t=\arcosh(m/R)$, $S=\sqrt{m^2-R^2}$, and $\Phi=mt-S$. Then
\begin{equation}\label{eq:J-saddle}
 J_m^{(q)}(R)=(1+o(1))
       \frac{(\sinh t)^q}{\sqrt{2\pi S}}e^{-\Phi}.
\end{equation}
\end{lemma}
\begin{proof}
Differentiate \eqref{eq:J-integral} $q$ times and shift the contour to
$\theta=u+it$. The integrand is entire and periodic, so the vertical
sides cancel. Using $R\cosh t=m$ and $R\sinh t=S$ gives
\begin{multline}\label{eq:J-shift}
 J_m^{(q)}(R)=\frac{(\sinh t)^q e^{-\Phi}}{2\pi}
 \int_{-\pi}^{\pi}
   (\cos u-i\coth t\sin u)^q\\
 {}
   \times e^{S(\cos u-1)}e^{im(u-\sin u)}\dd u.
\end{multline}
Set $v=\sqrt S\,u$. The hypotheses imply $S\to\infty$ and
\[
 \frac{m}{S^{3/2}}\to0,\qquad
 \frac{\coth t}{\sqrt S}=\frac m{S^{3/2}}\to0.
\]
For each fixed $v$, the differentiated amplitude tends to one, and
$|m(u-\sin u)|\le m|u|^3/6=m|v|^3/(6S^{3/2})\to0$.
Thus the scaled integrand tends to $e^{-v^2/2}$. Since
$1-\cos u\ge2u^2/\pi^2$ for $|u|\le\pi$, its modulus is bounded,
after this substitution, by
$C(1+|v|)^q e^{-2v^2/\pi^2}$, uniformly for large $m$.
Dominated convergence makes the integral in \eqref{eq:J-shift}
asymptotic to $\sqrt{2\pi/S}$. This proves \eqref{eq:J-saddle}.
\end{proof}

\begin{proposition}[An oscillatory coefficient estimate]
\label{prop:amplified}
Fix $B>0$ and an integer $k\ge1$, and define
$R_n'=n-BG_n$ and $h_k(x)=x^{2k}\cos x$. Then
\begin{equation}\label{eq:amplified}
 E_n(h_k;[-R_n',R_n'])\ge
 (1+o(1))\frac{(2B)^{k-1/4}}{\sqrt{2\pi}}
 n^{(4k-1)/3-c_0B^{3/2}}
 (\log n)^{(4k-1)/6}.
\end{equation}
\end{proposition}
\begin{proof}
Let $m$ be the smallest even integer greater than $n$. The functional
\[
 L(g)=\frac1{2\pi}\int_{-\pi}^{\pi}
                  g(R_n'\sin\theta)e^{im\theta}\dd\theta
\]
has norm at most one and annihilates every polynomial of degree at most
$n$. Differentiating \eqref{eq:J-integral} and using the evenness of $m$
gives
\[
 L(h_k)=(-1)^k(R_n')^{2k}J_m^{(2k)}(R_n').
\]
In Lemma~\ref{lem:J-saddle}, $m-R_n'=BG_n+O(1)$, so its hypotheses
hold. Since $R_n'\sinh t=S$, the absolute value of this coefficient is
\[
 (1+o(1))\frac{S^{2k-1/2}}{\sqrt{2\pi}}e^{-\Phi}.
\]
Here
\[
 S=(\sqrt{2B}+o(1))n^{2/3}(\log n)^{1/3},\qquad
 \Phi=c_0B^{3/2}\log n+o(1).
\]
The second estimate follows from \eqref{eq:Psi-expansion}; the next term
is $O(n^{-2/3}(\log n)^{5/3})$, and changing $m$ by $O(1)$ contributes
$o(1)$. Substitution proves \eqref{eq:amplified}.
\end{proof}

\begin{proof}[Proof of Theorem~\ref{thm:main}: necessity]
Normalize $\sigma=1$. If $\gamma_n/G_n$ does not tend to infinity, it is
bounded by some finite $B>0$ along an infinite subsequence. Choose a fixed
integer $k$ such that
\[
 \frac{4k-1}{3}>c_0B^{3/2},
\]
and define
\begin{equation}\label{eq:fk}
 F_k(x)=(2k+1)(1+x^2)^{k+1}+x^{2k}\cos x.
\end{equation}
For real $x$,
\[
 |(x^{2k}\cos x)''|\le(4k^2+2k+1)(1+x^2)^k,
\]
where each of $|x|^{2k}$, $|x|^{2k-1}$, and $|x|^{2k-2}$ is at most
$(1+x^2)^k$. The polynomial term in \eqref{eq:fk} has second derivative
at least $2(k+1)(2k+1)(1+x^2)^k$. Consequently,
\begin{equation}\label{eq:fk-convex}
 F_k''(x)\ge(4k+1)(1+x^2)^k>0.
\end{equation}
It is an even entire function of type one and polynomial real growth,
so $\Ac_1(F_k)=0$. For $n\ge2k+2$, subtracting its polynomial term does
not affect unconstrained degree-$n$ approximation. On the chosen
subsequence, $R_n\ge R_n'$, and therefore
\[
 \Ec_n(F_k;[-R_n,R_n])
 \ge E_n(h_k;[-R_n',R_n'])\longrightarrow\infty
\]
by Proposition~\ref{prop:amplified}. This disproves (ii). For general
$\sigma$, the function $F_k(\sigma x)$ has the required properties.
\end{proof}

\begin{example}[A counterexample of quartic growth]\label{ex:quartic}
The function
\[
 f(x)=x^4+2x^2+x^2\cos x
\]
is even, entire of type one, and satisfies
\[
 f''(x)=12x^2+4+(2-x^2)\cos x-4x\sin x
 \ge11x^2-4|x|+2\ge\frac{18}{11}.
\]
Nevertheless, at $R_n=n-G_n$,
\begin{equation}\label{eq:quartic}
 \Ec_n(f;[-R_n,R_n])\ge
 (1+o(1))\frac{2^{3/4}}{\sqrt{2\pi}}
 n^{1-c_0}\sqrt{\log n}\longrightarrow\infty,
 \qquad \Ac_1(f)=0.
\end{equation}
Indeed, subtract $x^4+2x^2$ and apply
Proposition~\ref{prop:amplified} with $k=B=1$. The lower bound in
\eqref{eq:quartic} diverges because $c_0<1$.
\end{example}

\begin{remark}[Dependence on the growth order]\label{rem:growth-quantifiers}
The lower bounds also hold for unconstrained approximation. For every
fixed $B>0$, increasing the polynomial growth order produces a target
for which the margin $\gamma_n=BG_n$ fails. Thus one margin for all
finite growth orders must satisfy the stronger condition in
Theorem~\ref{thm:main}.

In particular, the fixed-coefficient condition stated for the
polynomial-growth class in formula (1.12) of the author version of
\cite{Ganzburg2021}
(\href{https://arxiv.org/abs/2103.09376v2}{arXiv:2103.09376v2}) requires
an accompanying growth restriction. Example~\ref{ex:quartic} gives a
counterexample at $\gamma_n=G_n$.
\end{remark}

\section{Convex Bernstein constants and extremizers}\label{sec:powers}

Throughout this section, $\phi_\lambda(x)=|x|^\lambda$ and $\lambda\ge1$.
The algebraic-rate estimate first supplies an entire convex approximant
at finite distance from $|x|^\lambda$. The recovery theorem then
identifies the normalized limit through homogeneity.

\begin{proof}[Proof of Theorem~\ref{thm:bernstein}]
The classical Bernstein
estimate and the comparison theorem for algebraic rates of convex approximation
\cite{Bernstein1938,KLS2005} give
\begin{equation}\label{eq:power-rate}
 \sup_{n\ge1}n^\lambda\Ec_n(\phi_\lambda;[-1,1])<\infty.
\end{equation}
More precisely, \cite[Theorem~1.1]{KLS2005} transfers every positive
algebraic rate of unconstrained approximation of a continuous convex
function to convex approximation. That reference uses degree at most
$n-1$; an index shift gives \eqref{eq:power-rate} with our convention.

Choose convex $p_n\in\Pn$ with error within $o(n^{-\lambda})$ of the
infimum and set $P_n(z)=n^\lambda p_n(z/n)$. Then $P_n$ is convex on
$[-n,n]$ and
\[
 \norm{P_n-\phi_\lambda}{[-n,n]}
       =n^\lambda\norm{p_n-\phi_\lambda}{[-1,1]}\le C.
\]
Lemma~\ref{lem:compactness} produces a limit in $\Bc_1$ at finite
uniform distance from $\phi_\lambda$. Thus
$\Ac_1(\phi_\lambda)<\infty$.

Apply Theorem~\ref{thm:main} at $\sigma=1$ with $\gamma_n=\sqrt n$.
For $R>0$, homogeneity gives
\begin{equation}\label{eq:power-homogeneity}
 \Ec_n(\phi_\lambda;[-R,R])
      =R^\lambda\Ec_n(\phi_\lambda;[-1,1]).
\end{equation}
Since $R_n=n-\sqrt n$ satisfies $R_n/n\to1$, equations
\eqref{eq:recovery} and \eqref{eq:power-homogeneity} prove
\eqref{eq:bernstein}. Proposition~\ref{prop:attainment} gives attainment.
Replacing an extremizer $F$ by $(F(z)+F(-z))/2$ preserves convexity,
type, and its error bound, giving an even extremizer.

For a positive even integer $\lambda$, the target is a convex polynomial,
so the errors eventually vanish. Conversely, if the attained error is
zero, an entire function agrees with $|x|^\lambda$ on $\R$. Its order of
vanishing at zero forces $\lambda$ to be an integer, and agreement on
the negative axis forces that integer to be even. This proves positivity
in the remaining cases.
\end{proof}

\begin{corollary}[Type scaling and exact type]\label{cor:power-type}
For $\lambda\ge1$ and $\sigma>0$,
\begin{equation}\label{eq:type-scaling}
 \Ac_\sigma(\phi_\lambda)=\sigma^{-\lambda}\betac_\lambda.
\end{equation}
If $\lambda$ is not a positive even integer, every extremizer in
\eqref{eq:type-scaling} has type exactly $\sigma$. Moreover, the
classical unconstrained Bernstein constant
\[
 \beta_\lambda=\lim_{n\to\infty}
           n^\lambda E_n(\phi_\lambda;[-1,1])
\]
satisfies $\beta_\lambda\le\betac_\lambda$.
\end{corollary}
\begin{proof}
The bijection $F(z)\mapsto\sigma^{-\lambda}F(\sigma z)$ from
$\Bc_1$ to $\Bc_\sigma$ multiplies the error from $\phi_\lambda$ by
$\sigma^{-\lambda}$. If an extremizer had type less than $\sigma$, it
would belong to $\Bc_\tau$ for some $0<\tau<\sigma$, contradicting
\eqref{eq:type-scaling} and $\betac_\lambda>0$. The last assertion follows
by inclusion of the constrained polynomial class in the unconstrained
class, or of the corresponding entire-function classes.
\end{proof}

\begin{corollary}[Limits of rescaled near-minimizers]\label{cor:near-minimizers}
Suppose $p_n\in\Pn$ is convex on $[-1,1]$ and
\[
 \norm{\phi_\lambda-p_n}{[-1,1]}
 =\Ec_n(\phi_\lambda;[-1,1])+o(n^{-\lambda}).
\]
Then $P_n(z)=n^\lambda p_n(z/n)$ is relatively compact for local uniform
convergence in $\C$, and every subsequential limit is an extremizer in
\eqref{eq:bernstein}. If the $p_n$ are even, so are their limits.
\end{corollary}
\begin{proof}
Theorem~\ref{thm:bernstein} gives a common bound on
$\norm{P_n-\phi_\lambda}{[-n,n]}$. Apply
Lemma~\ref{lem:compactness}. The error of every limit is at most
$\betac_\lambda$, and by \eqref{eq:bernstein} it cannot be smaller.
\end{proof}

\section{Stretched-exponential growth}\label{sec:stretch}

We now prove Theorem~\ref{thm:stretch}. On the ellipse used in
Proposition~\ref{prop:poly-upper}, the decrease in $|\Imag z|$ away from
the imaginary axis controls the horizontal growth in
\eqref{eq:stretchstrip}. This gives the following estimate.

\begin{proposition}[A quantitative estimate for stretched-exponential growth]
\label{prop:stretch-upper}
Let $F\in\Bc_\sigma$ satisfy \eqref{eq:stretch-growth}. Assume
\eqref{eq:margin} and $\gamma_n\to\infty$, and let $D_n$ be defined by
\eqref{eq:saddle-parameters}. Then
\begin{multline}\label{eq:stretch-upper}
 \log\Ec_n(F;[-R_n,R_n])\le -D_n
 +C_1(n\gamma_n)^{\alpha/2}\\
 {}+C_2\left(\frac{n^3}{\gamma_n}\right)^{\alpha/[2(2-\alpha)]}
 +O(\log n),
\end{multline}
where $C_1,C_2>0$ are independent of $n$ and $\log0=-\infty$. Under
\eqref{eq:stretch-margin}, it follows that
\begin{equation}\label{eq:stretch-decay}
 \Ec_n(F;[-R_n,R_n])
 \le\exp\left[-(c_0+o(1))\frac{\gamma_n^{3/2}}{\sqrt n}\right].
\end{equation}
\end{proposition}
\begin{proof}
For $m,b,t_n$ as in \eqref{eq:saddle-parameters}, the ellipse
$\mathcal E(R_n,t_n)$ has semiaxes
\[
 U=R_n\cosh t_n=\frac m\sigma,\qquad
 V=R_n\sinh t_n=\frac{\sqrt{m^2-b^2}}\sigma.
\]
On this ellipse, $x^2/U^2+y^2/V^2=1$. Hence
\[
 \sigma|y|\le\sigma V-\frac{\sigma V}{2U^2}x^2.
\]
Lemma~\ref{lem:exactgrowth}(ii) gives
\[
 \log|F''(x+iy)|
 \le C+\sigma V+dV^\alpha
     +d|x|^\alpha-\frac{\sigma V}{2U^2}x^2.
\]
For $c>0$,
\[
 \sup_{u\ge0}(du^\alpha-cu^2)
 =\left(1-\frac\alpha2\right)d
      \left(\frac{\alpha d}{2c}\right)^{\alpha/(2-\alpha)}.
\]
Therefore, with $M_n=\max_{\mathcal E(R_n,t_n)}|F''|$,
\[
 \log M_n\le C+\sigma V+C_1V^\alpha
                       +C_2(U^2/V)^{\alpha/(2-\alpha)}.
\]
Apply \eqref{eq:ellipse-recovery}. Here
$U\asymp n$, $V\asymp\sqrt{n\gamma_n}$, and
$mt_n-\sigma V=D_n$, which proves \eqref{eq:stretch-upper}.

Let $p_\alpha=(1+\alpha)/(3-\alpha)$. The ratios of the two positive
power terms in \eqref{eq:stretch-upper} to
$\gamma_n^{3/2}/\sqrt n$ are constant multiples of
\[
 \left(\frac{n^{p_\alpha}}{\gamma_n}\right)^{(3-\alpha)/2},
 \qquad
 \left(\frac{n^{p_\alpha}}{\gamma_n}\right)^{(3-\alpha)/(2-\alpha)},
\]
respectively. They tend to zero under \eqref{eq:stretch-margin}. Moreover,
\[
 \frac{\gamma_n^{3/2}}{\sqrt n}
 =n^{2\alpha/(3-\alpha)}
   \left(\frac{\gamma_n}{n^{p_\alpha}}\right)^{3/2},
\]
so this quantity also dominates $\log n$.
Equation~\eqref{eq:Psi-expansion} gives \eqref{eq:stretch-decay}.
\end{proof}

\begin{proof}[Proof of Theorem~\ref{thm:stretch}]
When $\Ac_\sigma(f)<\infty$, every finite-error entire competitor
inherits an envelope of the form \eqref{eq:stretch-growth}, with the
same $\alpha$. Proposition~\ref{prop:stretch-upper} and the triangle
inequality give the limsup bound by $\Ac_\sigma(f)$.
Corollary~\ref{cor:weighted-liminf} gives the reverse bound, and also
handles the case $\Ac_\sigma(f)=\infty$.
\end{proof}

\begin{corollary}[A common margin for stretched-exponential growth]
\label{cor:stretch-choice}
The sequence $\gamma_n=n/\log n$ satisfies the hypothesis of
Theorem~\ref{thm:stretch} for every fixed $0<\alpha<1$. More generally,
$\gamma_n=n^\beta$ satisfies that hypothesis whenever
$(1+\alpha)/(3-\alpha)<\beta<1$.
\end{corollary}
\begin{proof}
For every fixed $p<1$, $n^{1-p}/\log n\to\infty$.
\end{proof}

The constants in these estimates depend on the envelope parameters
$\alpha,a,C$. For example, $\gamma_n=n^{2/3}$ is sufficient when
$\alpha=1/2$, since $(1+\alpha)/(3-\alpha)=3/5$.
Appendix~\ref{app:weights} extends
the argument to envelopes that grow faster than every fixed stretched
exponential.

\section{Arbitrary subexponential growth}\label{sec:subexp}

We now prove Theorem~\ref{thm:dichotomy}. Approximation on intervals
$[-\rho n/\sigma,\rho n/\sigma]$ with fixed $\rho<1$ gives a
sequence of margins for each target with finite entire-function error.
For a prescribed sequence of margins, we construct a counterexample
using a complex coefficient estimate, Baire's theorem, and an entire
majorant that ensures convexity.

\subsection{Approximation on smaller intervals}

\begin{lemma}[Approximation on smaller intervals]\label{lem:fixedfraction}
Suppose $F\in\B_\sigma$ satisfies \eqref{eq:subexp-growth}. For each
fixed $0<\rho<1$, there are real polynomials $q_n\in\Pn$ and constants
$C,c>0$ such that
\begin{equation}\label{eq:fixedfraction}
 \norm{F-q_n}{[-\rho n/\sigma,\rho n/\sigma]}\le C e^{-cn},
 \qquad q_n\longrightarrow F\text{ locally uniformly in }\C.
\end{equation}
If $F$ is convex on $\R$, the $q_n$ can be chosen convex on their
respective intervals.
\end{lemma}
\begin{proof}
For every $\varepsilon>0$ and every fixed $j\ge0$, the
Phragm\'en--Lindel\"of theorem in the four quadrants, followed by Cauchy's
formula, gives
\begin{equation}\label{eq:subexpstrip}
 |F^{(j)}(x+iy)|\le C_{\varepsilon,j}
                  e^{\varepsilon|x|+(\sigma+\varepsilon)|y|}.
\end{equation}
For example, in the first quadrant use
$F(z)e^{-\varepsilon z+i(\sigma+\varepsilon)z}$; the real and imaginary
boundary estimates follow with the smaller allowance $\varepsilon/2$.
The quadrant theorem applies since its order is at most one.
Choose a fixed small $t>0$ with $\rho\sinh t<t$, and then choose
$\varepsilon>0$ so small that
\[
 \rho\sinh t+\frac{\varepsilon\rho}{\sigma}(\cosh t+\sinh t)<t.
\]
Chebyshev truncation and \eqref{eq:subexpstrip} give exponential recovery
on the indicated intervals. For convex $F$, apply the truncation to
$F''$ with degree $n-2$ and use Lemma~\ref{lem:repair}; the additional
factor of order $n^2$ is absorbed by the strict exponential gap.

These truncations also converge on every fixed complex disk $K$.
Indeed, for $R\asymp n$ and $z\in K$,
$|T_k(z/R)|\le e^{C_K k/R}$. The corresponding Chebyshev tail is bounded by
\[
 \sum_{k>n}2M_n e^{-k(t-C_K/R)}
 =\frac{2M_n e^{-(n+1)(t-C_K/R)}}{1-e^{-(t-C_K/R)}},
\]
where the ellipse bound $M_n$ satisfies $\log M_n\le n(t-c)$ for some
fixed $c>0$. Thus the tail still tends to zero exponentially.
In the convex construction the derivative
truncations converge locally, the nonnegative correction tends to zero,
and integration with matched initial values yields local convergence of
$q_n$ to $F$.
\end{proof}

\begin{lemma}[A liminf inequality]\label{lem:finite-liminf}
Let $f$ be continuous and convex, satisfy \eqref{eq:subexp-growth}, and
have $\Ac_\sigma(f)<\infty$. Every sequence \eqref{eq:margin} satisfies
\begin{equation}\label{eq:finite-liminf}
 \Ac_\sigma(f)\le\liminf_n\Ec_n(f;[-R_n,R_n]).
\end{equation}
\end{lemma}
\begin{proof}
Choose $F_0\in\Bc_\sigma$ at finite error from $f$; it inherits
\eqref{eq:subexp-growth}. If the liminf is finite, choose nearly optimal
convex $p_n$ on a subsequence realizing it. Fix $\varepsilon>0$ and let
$N_n=\lceil(1+\varepsilon)n\rceil$. Apply
Lemma~\ref{lem:fixedfraction} at degree $N_n$, using a fraction strictly
between $1/(1+\varepsilon)$ and $1$. It gives polynomials $q_n$ with
\[
 \deg q_n\le N_n,\qquad
 \norm{F_0-q_n}{[-R_n,R_n]}\to0,\qquad q_n\to F_0\text{ locally}.
\]
If $L$ is the finite liminf under consideration, then eventually
$\norm{p_n-q_n}{[-R_n,R_n]}\le K$, where
$K=L+\norm{f-F_0}{\R}+2$ is independent of $\varepsilon$.
Bernstein--Walsh yields
\[
 |p_n(z)-q_n(z)|\le K|\Phi_{R_n}(z)|^{N_n}.
\]
Using $\varepsilon=1$ first gives local boundedness and a subsequential
limit $P$ of $p_n$. For any other fixed $\varepsilon>0$, the same
subsequence and its corresponding $q_n$ imply
\[
 |P(z)-F_0(z)|\le K e^{(1+\varepsilon)\sigma|\Imag z|}.
\]
For each fixed $\varepsilon$ the comparison polynomials converge along
the same extracted subsequence; no further extraction is needed.
Let $\varepsilon\downarrow0$. Then $P-F_0$, and hence $P$, has type at
most $\sigma$. Convexity and the limiting error bound pass to $P$,
which proves \eqref{eq:finite-liminf}.
\end{proof}

The lemma gives the lower bound for every admissible margin sequence.
To prove Theorem~\ref{thm:dichotomy}(i), it remains to choose a margin
for which the upper bound agrees with the entire-function error.

\begin{proof}[Proof of Theorem~\ref{thm:dichotomy}(i)]
By Proposition~\ref{prop:attainment}, choose an extremizer $F_*$ for the
finite entire-function error. It inherits \eqref{eq:subexp-growth}.
Choose $0<\rho_1<\rho_2<\cdots<1$ with $\rho_j\to1$.
Lemma~\ref{lem:fixedfraction} provides increasing
integers $N_j$ such that for every $n\ge N_j$ the convex recovery error
for $F_*$ on $[-\rho_jn/\sigma,\rho_jn/\sigma]$ is at most $1/j$.
Choose $N_{j+1}>N_j$ and enlarge these thresholds so that
$(1-\rho_j)N_j\ge j$. On $N_j\le n<N_{j+1}$, define
$\gamma_n=(1-\rho_j)n$; for $1\le n<N_1$, define
$\gamma_n=(1-\rho_1)n$. Then $\gamma_n/n$ is nonincreasing to zero,
and $\gamma_n\ge j$ on the $j$th block, so $\gamma_n\to\infty$.
The triangle inequality gives the desired limsup bound, and
Lemma~\ref{lem:finite-liminf} gives the liminf bound.
\end{proof}

\subsection{A complex coefficient estimate}
We now prove Theorem~\ref{thm:dichotomy}(ii). Fix a sequence
\eqref{eq:margin}, and first assume $\sigma=1$, so $R_n=n-\gamma_n$.
The next estimate supplies the test functions for the Baire argument;
scaling at the end of the construction gives arbitrary $\sigma>0$.

For an integer $m\ge1$ and $R>0$, define, for
$g\in C([-R,R];\C)$,
\begin{equation}\label{eq:coefficient-functional}
 L_{m,R}(g)=\frac1\pi\int_0^\pi g(R\cos\theta)\cos(m\theta)\dd\theta.
\end{equation}
It has norm at most one and annihilates polynomials of degree less than
$m$. The standard degree-$m$ Chebyshev coefficient of $g(R\,\cdot)$
is $2L_{m,R}(g)$. For even $m$, the integral formula for the modified Bessel function
\cite[\href{https://dlmf.nist.gov/10.32.E3}{Eq.~10.32.3}]{DLMF2026} gives
\begin{equation}\label{eq:I-complex}
 L_{m,R}(\cosh((a+i)x))=I_m((a+i)R).
\end{equation}
Let $E_n^{\C}$ denote best uniform approximation by complex polynomials
of degree at most $n$ on the indicated interval.

\begin{lemma}[A complex coefficient lower bound]\label{lem:complex-obstruction}
For all sufficiently small fixed $a>0$, if $R_n/n\to1$, then
\begin{equation}\label{eq:complex-obstruction}
 \liminf_{n\to\infty}\frac1n
 \log E_n^{\C}(\cosh((a+i)x);[-R_n,R_n])>0.
\end{equation}
\end{lemma}
\begin{proof}
Fix $a>0$, and let $m_n$ be the smallest even integer exceeding $n$.
The points $(a+i)R_n/m_n$ tend to $a+i$ and lie in a compact subset
of the open right half-plane. The large-order
expansion \cite[\href{https://dlmf.nist.gov/10.41.E3}{Eq.~10.41.3},
\href{https://dlmf.nist.gov/10.41.E7}{Eq.~10.41.7}, and
\href{https://dlmf.nist.gov/10.41.iii}{Sect.~10.41(iii)}]{DLMF2026}
gives
\begin{equation}\label{eq:I-complex-asymptotic}
 \log|I_{m_n}((a+i)R_n)|
 =m_n\Real\eta\left(\frac{(a+i)R_n}{m_n}\right)+O(\log m_n),
\end{equation}
where
\begin{equation}\label{eq:eta}
 \eta(z)=\sqrt{1+z^2}
          +\log\frac{z}{1+\sqrt{1+z^2}}.
\end{equation}
The branches are continued from the positive real axis into the right
half-plane. Since $\eta'(z)=\sqrt{1+z^2}/z$, expansion from the right
at $z=i$ gives
\[
 \eta(i+a)=\frac{i\pi}{2}+\frac23(1-i)a^{3/2}+O(a^{5/2}).
\]
In particular $\Real\eta(i+a)>0$ for every sufficiently small fixed
$a>0$. Equation~\eqref{eq:I-complex-asymptotic}, $R_n/m_n\to1$, and
\eqref{eq:I-complex} prove \eqref{eq:complex-obstruction}.
\end{proof}

\subsection{A Baire category argument}
Let $\mathcal X$ be the real Fr\'echet space of real entire functions of
exponential type zero, with increasing norms
\[
 p_j(H)=\sup_{z\in\C}|H(z)|e^{-|z|/j},\qquad j=1,2,\ldots.
\]
These norms are finite precisely on this class. If $(H_\ell)$ is
Cauchy in every $p_j$, then for each $j$ the functions
$H_\ell(z)e^{-|z|/j}$ converge uniformly on $\C$. The limits agree
on compact sets after multiplication by $e^{|z|/j}$ and give an entire
function $H$ real on $\R$. Since $p_j(H_\ell-H)\to0$ for every
$j$, the limit belongs to $\mathcal X$. This proves completeness.
Define $H_e(z)=(H(z)+H(-z))/2$ and $H_o(z)=(H(z)-H(-z))/2$, and let
\[
 s_n(H)=E_n(H_e\cos;[-R_n,R_n])+E_n(H_o\sin;[-R_n,R_n]).
\]
These are continuous seminorms: distance to the linear space $\Pn$ is a
seminorm, and restriction to each fixed interval is continuous in the
norms $p_j$.

If $\sup_n s_n(H)<\infty$ for every $H\in\mathcal X$, Baire's theorem
would imply that for some $j$ and $C$,
\begin{equation}\label{eq:baire-bound}
 s_n(H)\le C p_j(H)\qquad(n\ge1,\ H\in\mathcal X).
\end{equation}
Indeed, the closed sets
$\{H:\sup_n s_n(H)\le k\}$ would cover $\mathcal X$, so one would have
nonempty interior. Translation and subadditivity then give a common bound in a
neighborhood of zero, and the increasing norms and homogeneity give
\eqref{eq:baire-bound}.

Choose a fixed sufficiently small $0<a<1/j$, and consider the polynomials
\[
 H_K(z)=\sum_{r=0}^K\frac{(az)^r}{r!}.
\]
They satisfy $p_j(H_K)\le1$. For each fixed $n$, let $K\to\infty$ in
\eqref{eq:baire-bound}, using uniform convergence on the interval. It
would follow that
\[
 E_n(\cosh(ax)\cos x;[-R_n,R_n])+
 E_n(\sinh(ax)\sin x;[-R_n,R_n])\le C
\]
for every $n$. The sum bounds the complex error of
$\cosh((a+i)x)=\cosh(ax)\cos x+i\sinh(ax)\sin x$, contrary to
Lemma~\ref{lem:complex-obstruction}.

Consequently there is a fixed $H\in\mathcal X$ for which $s_n(H)$ is
unbounded. At least one of the even functions
\begin{equation}\label{eq:hard-amplitude}
 k(z)=H_e(z)\cos z,\qquad k(z)=H_o(z)\sin z
\end{equation}
has $\limsup_n E_n(k;[-R_n,R_n])=\infty$. Express this function as
$k=h\cos$ or $k=h\sin$, with $h$ of exponential type zero.

\subsection{Convexity of the counterexample}
We now add a function of type zero whose second derivative dominates
$|k''|$. Let $h(z)=\sum_{r\ge0}a_rz^r$, and define
\begin{equation}\label{eq:majorants}
 A(z)=\sum_{r\ge0}|a_r|z^r,
 \qquad B(z)=1+A(z+1)^2+A(1-z)^2.
\end{equation}
The function $A$ is entire of type zero. Indeed, Cauchy's coefficient
bounds on $|z|=2r$ imply
$A(r)\le2\max_{|z|=2r}|h(z)|$ for $r>0$, and the latter has arbitrarily
small exponential growth. Thus $B$ is also of type zero. It is even and
real, and on the real axis
\[
 B(x)\ge1,\qquad B(x)\ge A(|x|+1).
\]
For the second inequality use the appropriate term $A(x+1)^2$ or
$A(1-x)^2$ and $1+u^2\ge u$ for $u\ge0$.
Cauchy's formula on a unit circle centered at $x$ gives
\[
 |h(x)|\le A(|x|+1),\quad
 |h'(x)|\le A(|x|+1),\quad
 |h''(x)|\le2A(|x|+1),
\]
so $|k''(x)|\le5B(x)$. Let $Q$ be defined by
\begin{equation}\label{eq:Q}
 Q''=6B,\qquad Q(0)=Q'(0)=0.
\end{equation}
Equivalently,
\[
 Q(z)=6z^2\int_0^1(1-s)B(sz)\dd s.
\]
This formula gives $|Q(z)|\le3C_\varepsilon|z|^2e^{\varepsilon|z|}$
for every $\varepsilon>0$. Hence $Q$ is even, real entire, and of
exponential type zero. Let $F=Q+k$. Then
\begin{equation}\label{eq:convexified}
 F\in\B_1,\qquad F''(x)\ge B(x)\ge1.
\end{equation}
Since sine and cosine are bounded on $\R$, $F$ has subexponential
real-axis growth.

Furthermore $E_n(Q;[-R_n,R_n])\to0$ exponentially. To see this directly,
use any fixed Bernstein ellipse and the estimate
$|Q(z)|\le C_\varepsilon e^{\varepsilon|z|}$; choose $\varepsilon$ small
enough to leave a strict exponential gap in \eqref{eq:cheb}.
The triangle inequality for the unconstrained approximation seminorm
therefore gives
\[
 E_n(F;[-R_n,R_n])
 \ge E_n(k;[-R_n,R_n])-E_n(Q;[-R_n,R_n]),
\]
which is unbounded along a subsequence.
As $\Ec_n\ge E_n$ and $F\in\Bc_1$, its constrained errors are unbounded
and $\Ac_1(F)=0$.

For general $\sigma>0$, define
$\widetilde F(x)=\sigma^{-2}F(\sigma x)$ and restore
$R_n=(n-\gamma_n)/\sigma$. Then $\widetilde F''\ge1$, its type is at
most $\sigma$, and
\[
 \Ec_n(\widetilde F;[-R_n,R_n])
 =\sigma^{-2}\Ec_n(F;[-\sigma R_n,\sigma R_n]).
\]
The right-hand side is unbounded along a subsequence. If $\widetilde F$
had type smaller than $\sigma$, choose $0<\tau<\sigma$ with
$\widetilde F\in\Bc_\tau$ and $\tau/\sigma<\rho<1$.
Since $R_n\le n/\sigma<\rho n/\tau$,
Lemma~\ref{lem:fixedfraction} at type $\tau$ would give errors tending
to zero on $[-R_n,R_n]$, a contradiction. Thus the type is exactly
$\sigma$. This completes the proof of
Theorem~\ref{thm:dichotomy}(ii).

\begin{corollary}[Further properties of the counterexample]\label{cor:strong-obstruction}
The function in Theorem~\ref{thm:dichotomy}(ii) can be chosen so that
$\Ac_\tau(F)=\infty$ for every $0<\tau<\sigma$. Moreover, given any
prescribed $b_n\ge1$ with $\log b_n=o(n)$, it can be chosen with
\[
 \limsup_n\frac{\Ec_n(F;[-R_n,R_n])}{b_n}=\infty.
\]
\end{corollary}
\begin{proof}
A finite-error competitor of type $\tau<\sigma$ would inherit
subexponential growth; fixed-fraction recovery would then bound the
polynomial errors on $[-R_n,R_n]$, contradicting their unboundedness.
For the second assertion, first carry out the construction at $\sigma=1$
with the seminorms $s_n/b_n$.
Since $\log b_n=o(n)$, the exponential lower bound in
Lemma~\ref{lem:complex-obstruction}, divided by $b_n$, still tends
to infinity. The function $Q$ has exponentially small approximation
error, and $b_n\ge1$ gives
$E_n(Q;[-R_n,R_n])/b_n\to0$. The resulting convex function
therefore has the claimed normalized lower bound. The same rescaling
gives the assertion for arbitrary $\sigma>0$.
\end{proof}

\section{Dependence on the exponential type}\label{sec:type}

We next relate approximation on $[-n/\sigma,n/\sigma]$ to the dependence
of $\Ac_\sigma(f)$ on $\sigma$. The entire-function classes are nested,
so $\Ac_\sigma(f)$ is nonincreasing. Its one-sided limits are denoted by
\[
 \Ac_{\sigma-}(f)=\lim_{\tau\uparrow\sigma}\Ac_\tau(f),\qquad
 \Ac_{\sigma+}(f)=\lim_{\tau\downarrow\sigma}\Ac_\tau(f),
\]
where $\sigma>0$. Both limits exist in $[0,\infty]$.

\begin{proposition}[Right continuity]\label{prop:right-continuity}
For every real-valued target $f$, the function
$\sigma\mapsto\Ac_\sigma(f)$ is nonincreasing and right-continuous:
\begin{equation}\label{eq:right-continuity}
 \Ac_{\sigma+}(f)=\Ac_\sigma(f),\qquad \sigma>0.
\end{equation}
\end{proposition}
\begin{proof}
Monotonicity follows from inclusion of the entire-function classes.
Let $L=\Ac_{\sigma+}(f)\le\Ac_\sigma(f)$. If $L=\infty$, equality
follows by monotonicity. Otherwise choose $\sigma_j\downarrow\sigma$ and
$F_j\in\Bc_{\sigma_j}$ whose errors tend to $L$. These errors are bounded,
so $\norm{F_j-F_1}{\R}$ is bounded independently of $j$.
Lemma~\ref{lem:halfplane}, applied at type $\sigma_1$, makes the
differences locally bounded. Montel's theorem therefore gives a
subsequence converging locally to a real convex entire function $F$
with $\norm{f-F}{\R}\le L$. Fix $\varepsilon>0$ and choose $J$ such that
$\sigma_J<\sigma+\varepsilon$. For $j\ge J$, $F_j-F_J$ is bounded on
$\R$ by a common constant and has type at most $\sigma_J$.
Lemma~\ref{lem:halfplane}, followed by passage to the limit, shows that
$F-F_J$ has type at most $\sigma_J$. Thus $F$ has type at most
$\sigma+\varepsilon$. Since $\varepsilon$ is arbitrary, $F\in\Bc_\sigma$,
so $\Ac_\sigma(f)\le L$.
\end{proof}

\begin{theorem}[Bounds at the critical scale]\label{thm:critical}
Let $f$ be continuous and convex. Suppose either that $f$ has polynomial
real-axis growth, or that it has subexponential real-axis growth and
$\Ac_\sigma(f)<\infty$. Then
\begin{equation}\label{eq:critical-sandwich}
 \begin{split}
 \Ac_\sigma(f)
 &\le\liminf_n\Ec_n(f;[-n/\sigma,n/\sigma])\\
 &\le\limsup_n\Ec_n(f;[-n/\sigma,n/\sigma])
 \le\Ac_{\sigma-}(f).
 \end{split}
\end{equation}
In particular, if $\Ac_{\sigma-}(f)=\Ac_\sigma(f)$, the critical-scale
limit identity holds with no margin.
\end{theorem}
\begin{proof}
The lower bound is Corollary~\ref{cor:weighted-liminf} in the polynomial
case, and Lemma~\ref{lem:finite-liminf} in the other case, with
$\gamma_n=0$. For every $0<\tau<\sigma$ with finite error, choose a
finite-error type-$\tau$ competitor. It inherits the target's
subexponential growth, and Lemma~\ref{lem:fixedfraction} applies at the
fixed fraction $\tau/\sigma$. Take the infimum over competitors and then
let $\tau\uparrow\sigma$. If all these errors are infinite, the upper
bound is automatic.
\end{proof}

\begin{corollary}[Possible exceptional types]\label{cor:exceptional-types}
For a fixed continuous convex subexponential target, on every interval
where $\Ac_\sigma(f)$ is finite, failure of the critical-scale limit
identity can occur only at a left jump of this nonincreasing function.
There are at most countably many such types. The same assertion holds
for polynomial-growth targets, with the infinite-error cases already
settled by Corollary~\ref{cor:weighted-liminf}.
\end{corollary}
\begin{proof}
The set of positive types with finite error is an interval, possibly
empty, with at most one left endpoint. On its interior, a monotone
real-valued function has at most countably many discontinuities.
Theorem~\ref{thm:critical} gives the assertion there; inclusion of the
possible endpoint preserves countability. In the polynomial-growth case,
Corollary~\ref{cor:weighted-liminf} settles every type with infinite error.
\end{proof}

\section{Further questions}\label{sec:questions}

For targets of a prescribed polynomial growth order, a natural problem
is to determine the smallest margins that guarantee \eqref{eq:recovery}.
The family \eqref{eq:fk} gives lower bounds depending on that order,
while \eqref{eq:poly-upper} gives sufficient bounds. Matching these
bounds would determine the dependence on the growth order and the
constants in the margin condition.

For stretched-exponential growth, the corresponding question is whether
the exponent $(1+\alpha)/(3-\alpha)$ in Theorem~\ref{thm:stretch} can be
reduced. More generally, one may seek necessary and sufficient
conditions on a margin sequence for recovery under a prescribed
subexponential envelope. Such a criterion would connect the quantitative
estimates with the distinction in Theorem~\ref{thm:dichotomy} between
a margin chosen for each target and one prescribed for the whole class.

For the power-function problem, further questions concern uniqueness of
the convex entire extremizer and strict comparison with the
unconstrained Bernstein constant when $\lambda\ge1$ is not an even
integer. The compactness of near-minimizers also suggests studying
conditions under which the full sequence converges to an extremizer.
Quantitative estimates for the positive constants and rates of
convergence in \eqref{eq:bernstein} would complement the limit theorem.

\appendix
\section{A weighted extension}\label{app:weights}

The argument of Section~\ref{sec:stretch} extends to more general growth
envelopes through a weighted estimate for entire functions. Let
$g:\R\to[1,\infty)$ be even, locally bounded,
measurable, and submultiplicative:
$g(x+y)\le g(x)g(y)$. Suppose that
\begin{equation}\label{eq:BD}
 \int_1^\infty\frac{\log g(t)}{t^2}\dd t<\infty.
\end{equation}
For such weights this is the one-dimensional integral form of the
Beurling--Domar condition. Define
\begin{equation}\label{eq:Sg}
 S_g(y)=\frac1\pi\int_{-\infty}^{\infty}
          \frac{\log g(t)}{t^2+\max\{1,y\}^2}\dd t,
 \qquad y\ge0.
\end{equation}
By \cite[Theorem~3.3]{BSSS2024}, applied with $p=\infty$, every
$F\in\B_\sigma$ with $F/g$ bounded on $\R$ satisfies
\begin{equation}\label{eq:weighted-strip}
 |F''(x+iy)|\le C_F g(x)
             e^{\sigma|y|+|y|S_g(|y|)}.
\end{equation}
The directional indicator in that theorem is at most $\sigma$.
For each fixed $y$, its essential-supremum estimate holds for almost
every real $u$. Given $x$, choose such points $u_j\to x$ with
$|u_j-x|\le1$. Submultiplicativity gives
\[
 g(u_j)\le g(x)\sup_{|v|\le1}g(v).
\]
Continuity of $F''$ gives \eqref{eq:weighted-strip}, after enlarging
$C_F$ by a factor depending only on $g$.

These weights have subexponential growth on $\R$. Indeed, if
$h=\log g$, then \eqref{eq:BD} and $h\ge0$ imply
$\liminf_{t\to\infty}h(t)/t=0$. Given $\varepsilon>0$, choose $a>0$
with $h(a)\le\varepsilon a$. For $x\ge0$, the decomposition
$x=ka+r$, where $k\ge0$ is an integer and $0\le r<a$, gives
\[
 h(x)\le k h(a)+h(r)
       \le\varepsilon x+\sup_{0\le u\le a}h(u).
\]
Evenness gives the same bound for $x<0$.

\begin{proposition}[A sufficient weighted criterion]\label{prop:weighted}
Assume \eqref{eq:margin} and \eqref{eq:BD}. For $t>0$ set
\[
 \begin{split}
 H_n(t)=\sup_{0\le\theta\le2\pi}\bigl\{
 &\log g(R_n\cos\theta\cosh t)+\sigma Y_n(\theta,t)\\
 &+Y_n(\theta,t)S_g(Y_n(\theta,t))\bigr\},\qquad
 Y_n(\theta,t)=R_n|\sin\theta|\sinh t.
 \end{split}
\]
If there exist $t_n>0$ such that
\begin{equation}\label{eq:weighted-criterion}
 (n-2)t_n-H_n(t_n)
       -\log\frac{R_n^2}{e^{t_n}-1}\longrightarrow\infty,
\end{equation}
then every $F\in\Bc_\sigma$ with $F/g$ bounded on $\R$ satisfies
$\Ec_n(F;[-R_n,R_n])\to0$. Consequently, every continuous convex $f$
with $|f|\le Cg$ and $\Ac_\sigma(f)<\infty$ satisfies
\eqref{eq:recovery}.
\end{proposition}
\begin{proof}
The strip estimate \eqref{eq:weighted-strip} and
Lemma~\ref{lem:repair} give
\[
 \Ec_n(F;[-R_n,R_n])
 \le\frac{2C_F R_n^2}{e^{t_n}-1}
     \exp\bigl(H_n(t_n)-(n-2)t_n\bigr).
\]
The right-hand side tends to zero by \eqref{eq:weighted-criterion}.
If $\norm{f-F}{\R}<\infty$, then $|F|\le C'g$ on $\R$ because
$g\ge1$. The triangle inequality and the infimum over these
competitors give the upper bound in \eqref{eq:recovery}.
The subexponential bound for $g$ and Lemma~\ref{lem:finite-liminf}
give the lower bound.
\end{proof}

The weighted estimate also applies to
$\exp(a|x|/\log^2(e+|x|))$, which eventually exceeds every
$\exp(b|x|^\alpha)$ with fixed $a,b>0$ and $0<\alpha<1$.

\begin{proposition}[A logarithmic correction to exponential growth]\label{prop:near-exponential}
Let $f$ be continuous and convex, and suppose that, for some $C,a>0$,
\[
 |f(x)|\le C\exp\left(\frac{a|x|}{\log^2(e+|x|)}\right),
 \qquad \Ac_\sigma(f)<\infty.
\]
Under \eqref{eq:margin}, the condition
\begin{equation}\label{eq:near-exponential-margin}
 \frac{\gamma_n\log n}{n}\longrightarrow\infty
\end{equation}
is sufficient for \eqref{eq:recovery}. In particular,
$\gamma_n=n/\sqrt{\log n}$ satisfies this condition.
\end{proposition}
\begin{proof}
Let $g(x)=\exp(a|x|/\log^2(e+|x|))$.
The function $\omega(t)=t/\log^2(e+t)$ is increasing, and
$\omega(t)/t$ is decreasing. It is therefore subadditive on
$[0,\infty)$, so $g$ is submultiplicative on $\R$. Condition
\eqref{eq:BD} holds. Splitting the integral in \eqref{eq:Sg} at
$\sqrt y$ and $y$, for large $y$, gives
\begin{equation}\label{eq:Sg-log}
 S_g(y)\le\frac{C_a}{\log(e+y)}.
\end{equation}
The portions from $0$ to $\sqrt y$, from $\sqrt y$ to $y$, and from
$y$ to infinity are respectively $O(1/y)$, $O(1/\log^2y)$, and
$O(\int_y^\infty\dd t/(t\log^2t))=O(1/\log y)$.

Let $F\in\Bc_\sigma$ satisfy $\norm{f-F}{\R}<\infty$, and use
$t_n,D_n$ from \eqref{eq:saddle-parameters}. The ellipse
$\mathcal E(R_n,t_n)$ has semiaxes
\[
 U=\frac{n-2}{\sigma},\qquad
 V=\frac{\sqrt{(n-2)^2-(n-\gamma_n)^2}}\sigma.
\]
Equations \eqref{eq:weighted-strip} and \eqref{eq:Sg-log} imply
\[
 \log\max_{\mathcal E(R_n,t_n)}|F''|
 \le C+\sigma V+\frac{CU}{\log^2(e+U)}
                       +\frac{CV}{\log(e+V)}.
\]
By
\eqref{eq:near-exponential-margin}, $\delta_n\log n\to\infty$.
In particular $\log V\sim\log n$, and the positive terms beyond
$\sigma V$ are
\[
 O\left(\frac n{\log^2n}\right)
  +O\left(\frac{n\sqrt{\delta_n}}{\log n}\right)
       =o(n\delta_n^{3/2}).
\]
Since $D_n\sim c_0n\delta_n^{3/2}$ and
$n\delta_n^{3/2}/\log n\to\infty$, the estimate
\eqref{eq:ellipse-recovery} gives
$\Ec_n(F;[-R_n,R_n])\to0$.
Taking the infimum over competitors and applying
Lemma~\ref{lem:finite-liminf} proves the limit identity.
\end{proof}

More generally, for $g_\beta(x)=\exp\bigl(|x|/\log^\beta(e+|x|)\bigr)$, the integral
in \eqref{eq:BD} is finite precisely when $\beta>1$
\cite[Sect.~2]{BSSS2024}. For $\beta\le1$, Proposition~\ref{prop:weighted}
does not apply.

\section{Fixed fractions of the critical interval}\label{app:fractions}

For comparison with the shrinking-margin results, consider the assertion
\begin{equation}\label{eq:fixed-assertion}
 \limsup_{n\to\infty}
 \Ec_n(f;[-\rho n/\sigma,\rho n/\sigma])\le\Ac_\sigma(f),
 \qquad \rho>0.
\end{equation}
The distinction between unrestricted exponential growth and polynomial
real-axis growth is already visible at this fixed-fraction level. The
entire-function cutoff and the critical cosine obstruction are
classical; see \cite[Introduction]{Ganzburg2021}. We give proofs and
convex examples at the endpoints.

\subsection{Unrestricted exponential growth}

\begin{proposition}[A finite-difference lower bound]
\label{prop:finite-difference}
For $\tau,R>0$,
\begin{equation}\label{eq:finite-difference}
 E_n(e^{\tau x};[-R,R])
 \ge\left[\sinh\left(\frac{\tau R}{n+1}\right)\right]^{n+1}.
\end{equation}
Consequently $f(x)=e^{\sigma x}$ has $\Ac_\sigma(f)=0$ but
$\Ec_n(f;[-R_n,R_n])\to\infty$ for every sequence
\eqref{eq:margin}.
\end{proposition}
\begin{proof}
Let $m=n+1$ and $h=2R/m$. The $m$th difference on the nodes
$-R+jh$, $0\le j\le m$, annihilates every degree-$n$ polynomial, and its
norm as a functional is $2^m$. Applied to the exponential it equals
\[
 e^{-\tau R}(e^{\tau h}-1)^m
       =\left[2\sinh\left(\frac{\tau R}{m}\right)\right]^m.
\]
Divide by $2^m$. For $\tau=\sigma$ and
$R_n=(n-\gamma_n)/\sigma$, the base tends to $\sinh1>1$.
\end{proof}

The lower bound also applies when the target has type strictly below
the allowed value. For example, $f(x)=e^{0.95x}$ belongs to $\Bc_1$ with
strictly smaller type, yet its errors on $[-0.95n,0.95n]$ diverge by
\eqref{eq:finite-difference}, since $0.95^2>\arsinh1$.

Define
\begin{equation}\label{eq:rho0}
 \rho_0=\max_{t>0}\frac{t}{\cosh t}=0.662743419\ldots.
\end{equation}
The maximizing point is the unique $t_0>0$ satisfying
$t_0\tanh t_0=1$. The reciprocal $1/\rho_0=1.508879\ldots$ is the
classical Bernstein cutoff discussed in \cite[Introduction]{Ganzburg2021}.

\begin{lemma}[A positive coefficient lower bound]\label{lem:I-real}
For fixed $\rho>0$, with $\eta$ defined in \eqref{eq:eta}, there is
a constant $C_\rho>0$ such that
\begin{equation}\label{eq:I-real}
 I_{n+1}(\rho n)\ge C_\rho n^{-1}e^{n\eta(\rho)}
\end{equation}
for all sufficiently large $n$. Moreover
$\eta(\rho)<0$, $=0$, or $>0$ according as $\rho<\rho_0$, $=\rho_0$,
or $>\rho_0$.
\end{lemma}
\begin{proof}
The series
\[
 I_m(R)=\sum_{k=0}^\infty
            \frac{(R/2)^{m+2k}}{k!\,(m+k)!}
\]
has positive terms. Let $c=(\sqrt{1+\rho^2}-1)/2$ and take
$k=\lfloor cn\rfloor$ in the series for $m=n+1$.
Stirling's formula, with $c(1+c)=\rho^2/4$, gives for this term
\[
 (1+O(n^{-1}))\frac{e^{n\eta(\rho)}}{2\pi(1+c)n}.
\]
This proves \eqref{eq:I-real}. Direct minimization gives
\[
 \eta(\rho)=\min_{t>0}(\rho\cosh t-t),
 \qquad t=\arsinh(1/\rho).
\]
Its sign is therefore negative precisely when
$\rho<t/\cosh t$ for some $t$, which proves the assertion about
\eqref{eq:rho0}.
\end{proof}

\begin{proposition}[The range for unrestricted growth]
\label{prop:unrestricted-fraction}
Assertion \eqref{eq:fixed-assertion} holds for every continuous convex
$f$ if and only if $0<\rho<\rho_0$. For the exponential target,
\begin{equation}\label{eq:exp-rate}
 \lim_{n\to\infty}\frac1n
 \log\Ec_n(e^{\sigma x};[-\rho n/\sigma,\rho n/\sigma])
       =\eta(\rho).
\end{equation}
At the boundary $\rho=\rho_0$, the convex target
\[
 f(x)=\bigl((\sigma x)^2+2\bigr)e^{\sigma x}
\]
has $\Ac_\sigma(f)=0$ and constrained errors bounded below by $c n$.
\end{proposition}
\begin{proof}
After rescaling, assume $\sigma=1$. For $\rho<\rho_0$, choose $t>0$
and $\varepsilon>0$ such that $(1+\varepsilon)\rho\cosh t<t$.
Every $F\in\Bc_1$ satisfies
$|F''(z)|\le C_\varepsilon e^{(1+\varepsilon)|z|}$, so
Lemma~\ref{lem:repair}, with $R=\rho n$, gives
$\Ec_n(F;[-\rho n,\rho n])\le Cn^2e^{-cn}$ for some $c>0$.
The triangle inequality and the infimum over all $F\in\Bc_1$ with
$\norm{f-F}{\R}<\infty$ prove \eqref{eq:fixed-assertion} when
$\Ac_1(f)<\infty$. The assertion is immediate when $\Ac_1(f)=\infty$.

For every integer $m\ge1$, the functional in
\eqref{eq:coefficient-functional} satisfies $L_{m,R}(e^x)=I_m(R)$.
Taking $m=n+1$ and applying Lemma~\ref{lem:I-real} gives the lower
logarithmic rate in \eqref{eq:exp-rate}.
For the upper rate use $F''=e^z$, whose ellipse maximum is
$e^{R\cosh t}$, and take $t=\arsinh(1/\rho)$ in
\eqref{eq:ellipse-recovery}. The polynomial prefactor is of order
$n^2$, and the exponent is $n\eta(\rho)+O(1)$. This proves
\eqref{eq:exp-rate}, and excludes every $\rho>\rho_0$.

At $\rho=\rho_0$, take $f(x)=(x^2+2)e^x$; its second derivative is
$e^x(x+2)^2\ge0$, and $\Ac_1(f)=0$. For $R=\rho_0n$ and $m=n+1\ge3$,
the degree-$m$ Chebyshev coefficient of $f(Ry)$ is
\[
 R^2\left(I_m(R)+\frac{I_{m-2}(R)+I_{m+2}(R)}2\right)+4I_m(R).
\]
All terms are positive. The coefficient functional has norm at most two
in this normalization, so the approximation error is at least
$R^2I_m(R)/2\ge c n$, by \eqref{eq:I-real} and $\eta(\rho_0)=0$.
This excludes the boundary. Scaling $x$ to $\sigma x$ completes the proof.
\end{proof}

\subsection{Polynomial growth and the critical cosine estimate}
Critical-scale approximation of $\cos x$ gives a classical obstruction;
see \cite[Introduction]{Ganzburg2021}. The following proof gives the
uniform positive lower bound needed for the convex example below.

\begin{lemma}[Cosine at critical scale]\label{lem:cos-critical}
There exists $c_*>0$ such that
\begin{equation}\label{eq:cos-critical}
 E_n(\cos x;[-n,n])\ge c_*
\end{equation}
for all sufficiently large $n$.
\end{lemma}
\begin{proof}
We first estimate $J_m(n)$ for $m-n=O(n^{1/3})$. For $m\ge n$, let
$s=(m-n)/n^{1/3}$. From \eqref{eq:J-integral},
\[
 n^{1/3}J_m(n)
 =\frac1\pi\int_0^{\pi n^{1/3}}
 \cos\left(su+n\left(\frac{u}{n^{1/3}}
                       -\sin\frac{u}{n^{1/3}}\right)\right)\dd u.
\]
Uniformly for $s$ in a compact subset of $[0,\infty)$, this tends to
\[
 \mathcal A(s)=\frac1\pi\int_0^\infty
                                  \cos(su+u^3/6)\dd u.
\]
Indeed, on every finite interval Taylor's formula gives uniform
convergence. For the tail starting at $T$, the phase derivative is
increasing and at least $2u^2/\pi^2$; integration by parts bounds the
finite-$n$ tail by $\pi^2/T^2$. The limiting tail is at most $4/T^2$.
This proves uniform convergence and continuity of $\mathcal A$.
Furthermore
\[
 \mathcal A(0)=\frac{6^{1/3}}{3\pi}\Gamma(1/3)\cos(\pi/6)>0.
\]
Choose $\delta>0$ so small that, for all large $n$,
\begin{equation}\label{eq:J-positive-band}
 J_m(n)\ge\frac{\mathcal A(0)}{4n^{1/3}}
 \quad\text{whenever }n\le m\le n+\delta n^{1/3}.
\end{equation}

Let
\[
 L=\left\lfloor\frac{\delta n^{1/3}}8\right\rfloor,
 \qquad M=\left\lceil\frac n2+\frac{\delta n^{1/3}}4\right\rceil.
\]
The Fej\'er kernel
\[
 \mathcal F_L(u)=\sum_{|j|<L}\left(1-\frac{|j|}{L}\right)e^{iju}
              =\frac1L\left(\frac{\sin(Lu/2)}{\sin(u/2)}\right)^2
\]
is nonnegative and has normalized integral one. Set
$K_n(\theta)=\mathcal F_L(2\theta)\cos(2M\theta)$.
Its normalized $L_1[-\pi,\pi]$ norm is at most one, and all its Fourier
frequencies have absolute value greater than $n$, since
$2(M-L+1)>n$. Thus it annihilates $p(n\sin\theta)$ for $p\in\Pn$.
The Fourier expansion is
\[
 \cos(n\sin\theta)=\sum_{k\in\mathbb Z}
                         J_{2k}(n)e^{2ik\theta}.
\]
Expanding $K_n$ and using the symmetry of the Fej\'er weights together
with $J_{-2k}(n)=J_{2k}(n)$ gives
\begin{align*}
 \norm{\cos-p}{[-n,n]}
 &\ge\left|\frac1{2\pi}\int_{-\pi}^\pi
                 \cos(n\sin\theta)K_n(\theta)\dd\theta\right|\\
 &=\sum_{|j|<L}\left(1-\frac{|j|}{L}\right)J_{2(M+j)}(n)
 \ge\frac{\mathcal A(0)L}{4n^{1/3}}.
\end{align*}
All indices in the sum lie in the range
\eqref{eq:J-positive-band} for large $n$, and the weights sum to $L$.
Since $L/n^{1/3}\to\delta/8$, taking the infimum over $p$ proves
\eqref{eq:cos-critical}.
\end{proof}

\begin{proposition}[The range for polynomial growth]
\label{prop:polynomial-fraction}
Restricted to continuous convex targets of polynomial real-axis growth,
assertion \eqref{eq:fixed-assertion} holds for every such target if and
only if $0<\rho<1$. Failure at $\rho=1$ is witnessed by
\[
 f(x)=(\sigma x)^2+\cos(\sigma x),\qquad
 f''(x)\ge\sigma^2,\qquad \Ac_\sigma(f)=0.
\]
\end{proposition}
\begin{proof}
The positive assertion follows from
Lemma~\ref{lem:fixedfraction} applied to finite-error competitors. For
the displayed target and $n\ge2$, subtraction of the quadratic term is
a bijection of $\Pn$. Therefore, after scaling,
\[
 \Ec_n(f;[-\rho n/\sigma,\rho n/\sigma])
 \ge E_n(\cos x;[-\rho n,\rho n])
 \ge E_n(\cos x;[-n,n])\ge c_*
\]
whenever $\rho\ge1$. This proves the sharpness of the range.
\end{proof}

\small
\section*{Declarations}

\noindent\textbf{Author contributions.}
Both authors contributed equally to the paper.

\vspace{0.5mm}
\noindent\textbf{Funding declaration.}
CD was funded by an EPSRC post doctoral award.

\vspace{0.5mm}
\noindent\textbf{Competing interests.}
The authors have no competing interests to declare.

\vspace{0.5mm}
\noindent\textbf{Data availability.}
No datasets were generated or analysed in this theoretical study.

\vspace{0.5mm}
\noindent\textbf{AI declaration.}
The authors formulated the research questions, developed the main mathematical ideas, and directed the investigation. Discussions with ChatGPT 6 contributed suggestions for proof strategies and intermediate arguments, and assisted in identifying gaps and exploring counterexamples. The authors critically evaluated these suggestions, developed and verified the proofs. The authors take responsibility for the mathematical results and the final manuscript.


\begin{thebibliography}{99}
\setlength{\itemsep}{0pt}

\bibitem{BSSS2024}
Berger, D., Schilling, R.L., Shargorodsky, E., Sharia, T.:
\href{https://doi.org/10.4171/JST/509}{An extension of the Liouville theorem
for Fourier multipliers to sub-exponentially growing solutions}.
J. Spectr. Theory \textbf{14}(2), 665--695 (2024)

\bibitem{Bernstein1914}
Bernstein, S.N.:
\href{https://doi.org/10.1007/BF02401828}{Sur la meilleure approximation de
$|x|$ par des polyn\^omes de degr\'es donn\'es}.
Acta Math. \textbf{37}, 1--57 (1914)

\bibitem{Bernstein1938}
Bernstein, S.N.:
\href{https://www.mathnet.ru/eng/im3513}{Sur la meilleure approximation de
$|x|^p$ par des polyn\^omes de degr\'es tr\`es \'elev\'es}.
Izv. Akad. Nauk SSSR Ser. Mat. \textbf{2}(2), 169--190 (1938)

\bibitem{Boas1954}
Boas, R.P., Jr.: Entire Functions. Pure and Applied Mathematics,
vol.~5. Academic Press, New York (1954)

\bibitem{FoucartPowers2017}
Foucart, S., Powers, V.:
\href{https://doi.org/10.1093/imanum/drw014}{Basc: constrained approximation
by semidefinite programming}.
IMA J. Numer. Anal. \textbf{37}(2), 1066--1085 (2017)

\bibitem{Gal2009}
Gal, S.G.:
\href{https://doi.org/10.1142/7426}{Approximation by Complex Bernstein and
Convolution Type Operators}. Series on Concrete and Applicable
Mathematics, vol.~8. World Scientific, Hackensack (2009)

\bibitem{Gal2016}
Gal, S.G.:
\href{https://arxiv.org/abs/1505.02177v2}{On the Bernstein's constant in
convex approximation}. arXiv:1505.02177v2 (2016)

\bibitem{Ganzburg1991}
Ganzburg, M.I.:
\href{https://doi.org/10.1007/BF01060838}{Limit theorems for the best
polynomial approximations in the $L_\infty$ metric}.
Ukrainian Math. J. \textbf{43}(3), 299--305 (1991)

\bibitem{Ganzburg2008}
Ganzburg, M.I.:
\href{https://doi.org/10.1090/memo/0897}{Limit Theorems of Polynomial
Approximation with Exponential Weights}.
Mem. Amer. Math. Soc. \textbf{192}, no.~897 (2008)

\bibitem{Ganzburg2021}
Ganzburg, M.I.:
\href{https://doi.org/10.1007/s40315-021-00364-x}{Asymptotic behaviour of
the error of polynomial approximation of functions like
$|x|^{\alpha+i\beta}$}.
Comput. Methods Funct. Theory \textbf{21}(1), 73--94 (2021)

\bibitem{Ganzburg2023}
Ganzburg, M.I.:
\href{https://doi.org/10.1007/s10476-022-0185-z}{Sharp constants of
approximation theory. IV. Asymptotic relations in general settings}.
Anal. Math. \textbf{49}(1), 79--136 (2023)

\bibitem{Ganzburg2025}
Ganzburg, M.I.:
\href{https://doi.org/10.1016/j.jmaa.2025.129510}{Discretization theorems
for entire functions of exponential type}.
J. Math. Anal. Appl. \textbf{550}(1), article 129510 (2025)

\bibitem{KLPS2011}
Kopotun, K.A., Leviatan, D., Prymak, A., Shevchuk, I.A.:
\href{https://pages.cs.wisc.edu/~deboor/sat/papers/16/}{Uniform and
pointwise shape preserving approximation by algebraic polynomials}.
Surv. Approx. Theory \textbf{6}, 24--74 (2011)

\bibitem{KLS2005}
Kopotun, K.A., Leviatan, D., Shevchuk, I.A.:
\href{https://doi.org/10.4153/CJM-2005-049-6}{Convex polynomial
approximation in the uniform norm: conclusion}.
Canad. J. Math. \textbf{57}(6), 1224--1248 (2005)

\bibitem{KLS2006}
Kopotun, K.A., Leviatan, D., Shevchuk, I.A.:
\href{https://doi.org/10.1007/s10474-006-0010-3}{Coconvex approximation in
the uniform norm: the final frontier}.
Acta Math. Hungar. \textbf{110}(1--2), 117--151 (2006)

\bibitem{KLS2019}
Kopotun, K.A., Leviatan, D., Shevchuk, I.A.:
\href{https://doi.org/10.5802/smai-jcm.54}{Uniform and pointwise shape
preserving approximation (SPA) by algebraic polynomials: an update}.
SMAI J. Comput. Math. \textbf{S5}, 99--108 (2019)

\bibitem{Lubinsky2007}
Lubinsky, D.S.:
\href{https://doi.org/10.1007/s00365-005-0626-5}{On the Bernstein constants
of polynomial approximation}.
Constr. Approx. \textbf{25}(3), 303--366 (2007)

\bibitem{DLMF2026}
Olver, F.W.J., Olde Daalhuis, A.B., Lozier, D.W., Schneider, B.I.,
Boisvert, R.F., Clark, C.W., Miller, B.R., Saunders, B.V., Cohl, H.S.,
McClain, M.A. (eds.):
\href{https://dlmf.nist.gov/}{NIST Digital Library of Mathematical
Functions}. Release 1.2.8 of 15 September 2026

\bibitem{Shisha1965}
Shisha, O.:
\href{https://doi.org/10.2140/pjm.1965.15.667}{Monotone
approximation}. Pacific J. Math. \textbf{15}(2), 667--671 (1965)

\bibitem{Trefethen2019}
Trefethen, L.N.:
\href{https://doi.org/10.1137/1.9781611975949}{Approximation Theory and
Approximation Practice, Extended Edition}. SIAM, Philadelphia (2019)

\end{thebibliography}
\end{document}